\documentclass[11pt]{article}
\usepackage{pb-diagram}
\usepackage{amsmath,amsthm,amsfonts,amssymb,latexsym}
\usepackage{amsmath}
\usepackage{hyperref}
\usepackage{a4wide}
\usepackage[all]{xy}

\newtheorem{thm}{Theorem}[section]
\newtheorem{cor}[thm]{Corollary}

\newtheorem{defn}[thm]{Definition}

\newtheorem{rmk}[thm]{Remark}

\newtheorem{theorem}{Theorem}[section]
\newtheorem{definition}[theorem]{Definition}
\newtheorem{lemma}[theorem]{Lemma}
\newtheorem{remark}[theorem]{Remark}
\newtheorem{proposition}[theorem]{Proposition}

\newtheorem{example}[theorem]{Example}

\begin{document}
\title{BiHom-Akivis algebras}
\author{Sylvain Attan \thanks{D\'{e}partement de Math\'{e}matiques, Universit\'{e} d'Abomey-Calavi
01 BP 4521, Cotonou 01, B\'{e}nin. E.mail: syltane2010@yahoo.fr}}

\maketitle
\begin{abstract}
BiHom-Akivis algebras are introduced. The BiHom-commutator-BiHom-associator algebra of a regular BiHom-algebra  is a BiHom-Akivis algebra. It is shown that  BiHom-Akivis algebras can be obtained from Akivis algebras by twisting along two algebra  endomorphisms. It is pointed out that a BiHom-Akivis algebra associated to a regular BiHom-alternative algebra is a BiHom-Malcev algebra. 
\end{abstract}
{\bf 2010 Mathematics Subject Classification:} 17A30, 17D10, 17D99

{\bf Keywords:} Akivis algebra, BiHom-associative algebra, BiHom-Lie algebra, BiHom-Akivis algebra, BiHom-Malcev algebra.
\section{Introduction}
An Akivis algebra $(A, \{-,-\}, \{-, ·, -\})$ is a vector space $A$ together with a bilinear
skew-symmetric map $(x, y)\longrightarrow \{x, y\}$ and a trilinear map $(x, y, z)\longrightarrow \{x, y, z\}$ satisfying the following so-called Akivis identity  for all $x,y,z \in A:$
\begin{eqnarray}
    \circlearrowleft_{x,y,z}\{x,\{y,z\}\}=\circlearrowleft_{x,y,z}\{x,y,z\}-\circlearrowleft_{x,y,z}\{y,x,z\}
\label{Akiv00}
\end{eqnarray}
These algebras were introduced by M.A. Akivis (\cite{akiv1, akiv2,akiv3}) as a tool in the study of some aspects of web geometry and its connection with loop theory. They  were originally called "W-algebras" \cite{akiv3} and later, K.H. Hofmann and K. Strambach \cite{Hof1} introduced the term “Akivis algebras” for such algebraic objects.

The theory of Hom-algebras originated from Hom-Lie algebras introduced by J.T. Hartwig, D. Larsson, and S.D. Silvestrov in \cite{HAR1} in the study of quasi-deformations of Lie algebras of vector fields, including q-deformations of Witt algebras and Virasoro algebras. Generalizing the relation between Lie algebras and associative algebras, the notion of a Hom-associative algebra is introduced  in \cite{MAK3}, where it is shown that the commutator algebra (with the twisting map) of a Hom-associative algebra is a Hom-Lie algebra. By twisting defining identities, other Hom-type algebras such as Hom-alternative algebras, Hom-Jordan algebras \cite{MAK1,YAU3}, Hom-Novikov algebras \cite{YAU4}, or Hom-Malcev algebras \cite{YAU3} are introduced and discussed. Moving forward in the general theory of Hom-algebras, a study of "binary-ternary" Hom-algebras is initiated in \cite{issa1} by defining the class of Hom-Akivis algebras as a Hom-analogue of the class of Akivis algebras (\cite{akiv1, akiv2, Hof1})  which are a typical example of binary-ternary (see also \cite{attan2} and  \cite{Gap1} for other classes of binary-ternary Hom-algebra).

Generalizing  the approah in \cite{scig} the authors of \cite{GRAZIANI}  introduce BiHom-algebras, which are algebras where the identities defining the structure are twisted by two homomorphisms $\alpha$ and $\beta$.
This class of algebras can be viewed as an extension of the class of Hom-algebras since, when the two linear maps of a BiHom-algebra are the same, it reduces to a Hom-algebra. These
algebraic structures include BiHom-associative algebras, BiHom-Lie algebras and BiHom-bialgebras, BiHom-alternative algebras, BiHom-Jordan algebras, Bihom-Malev algebras $\dots$

As for BiHom-associative, BiHom-Lie, BiHom-alternative, BiHom-Jordan, Bihom-Malev algebras $\dots$ we consider in this paper a twisted version by two commuting linear maps of the Akivis identity which defines the so-called Akivis algebras. We call "BiHom-Akivis algebra" this twisted Akivis algebra. It is known \cite{akiv3} that the commutator-associator algebra of a nonassociative algebra is an Akivis algebra. The Hom-version is this result an be found in \cite{issa1}. This led us to consider ``non-BiHom-associative algebras'' i.e. BiHom-nonassociative algebras or nonassociative BiHom-algebras  and we point out that the BiHom-commutator-BiHom-associator algebra of a regular non-BiHom-associative algebra has a BiHom-Akivis structure.  Also the class of BiHom-Akivis algebras contains the one of BiHom-Lie algebras in the same way as the class of Akivis (resp. Hom-Akivis) algebras  contains the one of Lie (resp. Hom-Le) algebras.

The rest of the present paper is organized as follows. In Section 2 we recall basic definitions and  results about  BiHom-algebras. Here, we prove that any two of the three conditions left BiHom-alternative, right Bihom-alternative and BiHom-flexible in a regular BiHom-algebra, imply the third (Proposition \ref{propo1}, Proposition \ref{propo2} and Proposition \ref{propo3} ).  In Section 3, BiHom-Akivis algebras are considered. Two methods of producing BiHom-Akivis algebras are provided starting with either  a regular BHom-algebras (Theorem \ref{thm0}) or classical Akivis algebras along with twisting maps (Corollary \ref{corimp}). BiHom-Akivis algebras are shown to be closed under twisting by self-morphisms (Theorem \ref{thm1}). In Section 4, BiHom-Akivis algebras associated to a regular BiHom-alternative algebras are shown to be BiHom-Malcev algebras (these later algebraic objects are recently introduced \cite{chtioui1}). This could be seen as a generalization of the construction  of Malcev (resp. Hom-Malcev) algebras from alternative \cite{Malt1} (resp. Hom-alternative \cite{YAU3}) algebras. 

Throughout this paper, all vector spaces and algebras are meant over a ground field $\mathbb{K}$ of characteristic 0.
\section{Preliminaries}
In the sequel, a BiHom-algebra refers to a quadruple $(A, \mu, \alpha, \beta),$ where
$\mu : A\otimes A\longrightarrow A,$ $\alpha: A\longrightarrow A$ and $\beta : A\longrightarrow A$ are linear maps such that $\alpha\beta=\beta\alpha$. The composition of maps is denoted by concatenation for simplicity and the map
 $\tau: A^{\otimes 2} \longrightarrow A^{\otimes 2}$ denotes the twist isomorphism $\tau(a\otimes b)=b\otimes a.$ 
\begin{defn}
A BiHom-algebra $(A,\mu,\alpha,\beta)$ is said to be  regular if 
$\alpha$ and $\beta$ are bijective
and  multiplicative if $\alpha\circ\mu=\mu\circ\alpha^{\otimes 2}$ and $\beta\circ\mu=\mu\circ\beta^{\otimes 2}.$ 
\end{defn}
\begin{defn}\cite{GRAZIANI}
Let $(A,\mu,\alpha,\beta)$ be a BiHom-algebra.
\begin{enumerate}
\item A BiHom-associator of $A$  is the trilinear map $as_{\alpha,\beta}:A^{\otimes 3} \longrightarrow A$ defined by
\begin{equation}\label{BiHomAss}
    as_{\alpha,\beta}=\mu \circ (\mu\otimes \beta- \alpha \otimes \mu).
\end{equation}
In terms of elements, the map $as_{\alpha,\beta}$ is given by
$$ as_{\alpha,\beta}(x,y,z)=\mu(\mu(x,y),\beta(z))-\mu(\alpha(x),\mu(y,z)),\ \forall x,y,z \in A.$$
\item A BiHom-associative algebra \cite{GRAZIANI} is a multiplicative Bihom-algebra $(A,\mu,\alpha,\beta)$
satisfying the following BiHom-associativity condition:
\begin{equation}\label{BiHom ass identity}
    as_{\alpha,\beta}(x,y,z)=0,\ \text{for all}\ x,y,z \in A.\
\end{equation}
\end{enumerate}
\end{defn}
Note that if  $\alpha=\beta=Id $, then the BiHom-associator coincide with the usual associator denoted by $as(,,).$
Clearly, a Hom-associative algebra $(A,\mu,\alpha)$ can be regarded as a BiHom-associative
algebra $(A,\mu,\alpha,\alpha)$.
\begin{remark}
A non-BiHom-assoiative algebra is a BiHom-algebra $(A,\mu,\alpha,\beta)$ for which there exists $x,y,z\in A$ such that $as_{\alpha,\beta}(x,y,z)=0.$
\end{remark}
 \begin{example}\label{ex1}
 Let $(A, \mu)$ be the two-dimensional algebra with basis $(e_1, e_2 )$ and
multiplication given by
$$\mu(e_1,e_2)=\mu(e_2,e_2)=e_1$$
and all missing products are 0. Then $(A, \mu)$ is nonassociative since, e.g., $\mu(\mu(e_1,e_2),e_2) = e_1 \neq 0 = \mu(e_1,\mu(e_2,e_2 )).$
 Next, if we define for any $\lambda\neq -1,$ linear maps $\alpha_{\lambda},\beta_{\lambda} : A \longrightarrow A$ by\\ 
 $\alpha_{\lambda}(e_1)=(\lambda+1)e_1,\ \alpha_{\lambda}(e_2)=\lambda e_1+e_2$ and 
 $\beta_{\lambda}(e_1)=\frac{1}{\lambda+1}e_1,\ \beta_{\lambda}(e_2)=\frac{-\lambda}{\lambda+1} e_1+e_2,$ then $A_{\alpha_{\lambda},\beta_{\lambda}}=(A, \mu_{\alpha_{\lambda},\beta_{\lambda}}=\mu\circ(\alpha_{\lambda}\otimes \beta_{\lambda}), \alpha_{\lambda}, \beta_{\lambda})$ is a non BiHom-associative algebra where the non-zero products are  $\mu_{\alpha_{\lambda},\beta_{\lambda}}(e_1,e_2)=
 \mu_{\alpha_{\lambda},\beta_{\lambda}}(e_2,e_2)=(\lambda+1)e_1$ since e.g.
 $as_{\alpha_{\lambda}, \beta_{\lambda}}(e_1,e_2,e_2)=(\lambda+1)e_1
 \neq 0.$ Actually, $A_{\alpha_{\lambda},\beta_{\lambda}}$ is a regular BiHom-algebra with $\beta=\alpha^{-1}.$
 \end{example}
 Let recall the notion of BiHom-alternative and BiHom-flexible algebras.
\begin{defn}\cite{chtioui1}
Let $(A,\mu,\alpha,\beta)$ be a BiHom-algebra
\begin{enumerate}
\item  $(A,\mu,\alpha,\beta)$ is said to be a left BiHom-alternative  $($resp. right BiHom-alternative $)$ if its satisfies the left BiHom-alternative identity,
\begin{equation}\label{sa1}
as_{\alpha,\beta}(\beta(x),\alpha(y),z)+as_{\alpha,\beta}(\beta(y),\alpha(x),z)=0,
\end{equation}
respectively, the right BiHom-alternative identity,
\begin{equation}\label{sa2}
as_{\alpha,\beta}(x,\beta(y),\alpha(z))+as_{\alpha,\beta}(x,\beta(z),\alpha(y))=0,
\end{equation}
 for all $x,y,z \in A$.
A BiHom-alternative algebra \cite{chtioui1} is the one which is both a left and right BiHom-alternative algebra.
\item $(A,\mu,\alpha,\beta)$ is said to be BiHom-flexible  \cite{attan1} if its  satisfies the BiHom-flexible identity,
\begin{eqnarray}
as_{\alpha,\beta}(\beta^{2}(x),\alpha\beta(y),\alpha^{2}(z)) + as_{\alpha,\beta}(\beta^{2}(z),\alpha\beta(y),\alpha^{2}(x)) = 0\label{BiHomflex1}
\end{eqnarray}
for all $x,y,z\in A.$
\end{enumerate}
\end{defn}
Observe that when $\alpha=\beta=Id,$ a BiHom-alternative and a Bihom-flexible algebra reduce to an alternative and flexible algebra respectively.
\begin{rmk}\label{rmk1}
\begin{enumerate}
\item Any BiHom-associative algebra is a BiHom-alternative algebra. and 
BiHom-flexible.
\item It is proved that equations (\ref{sa1}), (\ref{sa2}) and (\ref{BiHomflex1}) are respectively equivalent to
\begin{eqnarray}
as_{\alpha,\beta}(\beta(x),\alpha(x),z)=0,\ 
as_{\alpha,\beta}(x,\beta(y),\alpha(z))=0 \mbox{ and }
as_{\alpha,\beta}(\beta^{2}(x),\alpha\beta(y),\alpha^{2}(x)) = 0\nonumber
\end{eqnarray}
for all $x,y\in A.$
\end{enumerate}
\end{rmk}
\begin{lemma}\cite{chtioui1}\label{lem1} Let
$(A,\mu,\alpha,\beta)$ be a regular BiHom-algebra. Then, $(A,\mu,\alpha,\beta)$ is a regular BiHom-alternative algebra.
 if and only if the function $as_{\alpha,\beta}(\beta^2\otimes\alpha\beta\otimes\alpha^2)$ is alternating.
 \end{lemma}
\begin{proposition}\label{propo1} Any regular BiHom-alternative algebra is BiHom-flexible.
\end{proposition}
\begin{proof}
 Follows by a direct by Lemma \ref{lem1}.
\end{proof}
\begin{proposition}\label{propo2} Any regular left  BiHom-alternative BiHom-flexible algebra is BiHom-alternative.
\end{proposition}
\begin{proof} Let $(A,\mu,\alpha,\beta)$ be a regular left  BiHom-alternative BiHom-flexible algebra. We need just to prove $as_{\alpha,\beta}(x,\beta(y),\alpha(y))=0$ since it is equivalent to (\ref{sa2}) (see Remark \ref{rmk1} ).
Now, let pick $x,y\in A$ then, we have:
\begin{eqnarray}
&&as_{\alpha,\beta}(x,\beta(y),\alpha(y))=as_{\alpha,\beta}(\beta^2(\beta^{-2}(x)),\alpha\beta(\alpha^{-1}(y)),\alpha^2(\alpha^{-1}(y)))\nonumber\\
&&=-as_{\alpha,\beta}(\beta^2(\alpha^{-1}(y)),\alpha\beta(\alpha^{-1}(y)),\alpha^2(\beta^{-2}(x)))\ \mbox{ ( by (\ref{BiHomflex1}) \ )} \nonumber\\
&&=-as_{\alpha,\beta}(\beta(\beta\alpha^{-1}(y)),\alpha(\beta\alpha^{-1}(y)),\alpha^2\beta^{-2}(x))=0\ \mbox{ ( by (\ref{sa1}) \ ).} \nonumber
\end{eqnarray}
Hence, $(A,\mu,\alpha,\beta)$ is right Bihom-alternative and therefore, it is BiHom-alternative.
 \end{proof}
 Similarly, we can prove 
 \begin{proposition}\label{propo3} Any regular right  BiHom-alternative BiHom-flexible algebra is BiHom-alternative.
\end{proposition}
\begin{definition} Let $(A,[-,-], \alpha,\beta)$ be a BiHom-algebra.
\begin{enumerate}
\item The BiHom-Jacobiator of $A$ is the trilinear map $J_{\alpha,\beta} : A^{\times 3}\rightarrow A$ defined as
\begin{eqnarray}
J_{\alpha,\beta}(x,y,z)=\circlearrowright_{x,y,z}[\beta^2(x),[\beta(y),\alpha(z)]]
\end{eqnarray}
\item  $(A,[-,-], \alpha,\beta)$ is said to be a BiHom-Lie algebra if
\begin{enumerate}
\item $[\beta(x),\alpha(y)]=-[\beta(y),\alpha(x)]$ (BiHom-skew-symmetry)
\item $A$ satisfies a BiHom-Jacobi identity i.e.
 \begin{eqnarray}
 J_{\alpha,\beta}(x,y,z)=0 \label{BiHomJac}
 \end{eqnarray}
 for all $x,y,z\in A.$
 \end{enumerate}
 \item  $(A,[-,-], \alpha,\beta)$ is said to be a BiHom-Malcev algebra if
\begin{enumerate}
\item $[\beta(x),\alpha(y)]=-[\beta(y),\alpha(x)]$ (BiHom-skew-symmetry)
\item $A$ satisfies a BiHom-Malcev identity i.e.
 \begin{eqnarray}
 J_{\alpha,\beta}(\alpha\beta(x),\alpha\beta(y),[\beta(x),\alpha(z)])=
 [J_{\alpha,\beta}(\beta(x),\beta(y),\beta(z)),\alpha^2\beta^2(x)]
 \label{BiHomMalc}
 \end{eqnarray}
 for all $x,y,z\in A.$
\end{enumerate}
\end{enumerate}
\end{definition}
\begin{remark}
\begin{enumerate}
\item If $\alpha=\beta=Id,$ a Bihom-Lie (resp. BiHom-Malcev )algebra 
reduces to a Lie (resp. Malcev) algebra.
\item Any BiHom-Lie algebra is a BiHom-Malcev algebra.
\end{enumerate}
\end{remark}
\section{BiHom-Akivis algebras: Constructions and exemples}
In this section we give the notion and an example of a BiHom-Akivis algebra. We provide some construction methods of BiHom-Akivis algebras (the
construction from non-BiHom-associative algebras and the one from Akivis algebras).
\begin{defn}
A BiHom-Akivis algebra is a quintuuple ($V, [-,-], [-,-,-], \alpha, \beta$), where $V$ is a vector space, $[-,-] : V \times V \rightarrow V$ a BiHom-skew-symmetric bilinear map, $[-,-,-] : V \times V \times V \rightarrow V$ a trilinear map and $\alpha : V \rightarrow V$ a linear map such that 
\begin{eqnarray}
J_{\alpha,\beta}(x,y,z)= \circlearrowright_{x,y,z} [x,y,z] - \circlearrowright_{x,y,z} [y,x,z]
\label{BiHAk}
\end{eqnarray}
for all $x,y,z$ in $V$ where $J_{\alpha,\beta}(x,y,z)=\circlearrowright_{x,y,z} [\beta^2(x),[\beta(y),\alpha(z)]]$ is the BiHom-Jacobiator of $(V,[-,-],\alpha,\beta).$\\

A BiHom-Akivis algebra ($V, [-,-], [-,-,-], \alpha, \beta$) is said multiplicative if $\alpha$ and $\beta$ are  endomorphisms with respect to $[-,-]$ and $[-,-,-]$.
\end{defn}
In analogy Akivis and Hom-Akivis cases, let call (\ref{BiHAk}) the {\it BiHom-Akivis identity}. \\
\begin{remark}
\begin{enumerate}
\item If $\alpha =\beta= {Id}_{V}$, the BiHom-Akivis identity (\ref{BiHAk}) is the usual Akivis identity (\ref{Akiv00}).
\item The BiHom-Akivis identity (\ref{BiHAk}) reduces to the BiHom-Jacobi identity (\ref{BiHomJac}), when $[x,y,z] = 0$, for all $x,y,z$ in $V$.
\end{enumerate}  
\end{remark}
The following result shows how one can get BiHom-Akivis algebras from non-BiHom-associative algebras.
\begin{thm}\label{thm0} Let $(A,\mu,\alpha,\beta)$ be a mutipliative regular BiHom-algebra. Then
the BiHom-commutator-BiHom-associator algebra of $(A,\mu,\alpha,\beta)$   is a multiplicative BiHom-Akivis algebra.
\end{thm}
\begin{proof}
Let ($A, \mu , \alpha $) be a multiplicative regular BiHom-algebra. For any $x,y,z$ in $A$, define the operations 
$ [x,y] := \mu (x,y) - \mu (\alpha^{-1}\beta(y),\alpha\beta^{-1}(x))$ and
$[x,y,z] := as_{\alpha,\beta} (\alpha^{-1}\beta^2(x),\beta(y),\alpha(z))$.
Then, by \cite{chtioui1} (Lemma 2.1), we have 
$$J_{\alpha,\beta}(x,y,z)=\circlearrowright_{x,y,z} as_{\alpha,\beta} (\alpha^{-1}\beta^2(x),\beta(y),\alpha(z))-\circlearrowright_{x,y,z} as_{\alpha,\beta} (\alpha^{-1}\beta^2(y),\beta(x),\alpha(z))$$
i.e.
$$J_{\alpha,\beta}(x,y,z) = \circlearrowright_{x,y,z} [x,y,z] - \circlearrowright_{x,y,z} [y,x,z].$$ Hence, $(A, [-,-],[-,-,-], \alpha,\beta)$ is a multiplicative BiHom-Akivis algebra.
\end{proof}
The BiHom-Akivis algebra constructed by Theorem \ref{thm0} is said {\it associated} (with a given regular BiHom-algebra).
\begin{example}
Consider the family of regular non-BiHom-associative algebras
$A_{\alpha_{\lambda}, \beta_{\lambda}}$ of Example \ref{ex1}. By 
Theorem \ref{thm0} if define $ [x,y] := \mu_{\alpha_{\lambda},\beta_{\lambda}} (x,y) - \mu_{\alpha_{\lambda},\beta_{\lambda}} (\alpha_{\lambda}^{-1}\beta_{\lambda}(y),\alpha_{\lambda}\beta_{\lambda}^{-1}(x))$ and
$[x,y,z] := as_{\alpha_{\lambda},\beta_{\lambda}} (\alpha_{\lambda}^{-1}\beta_{\lambda}^2(x),\beta_{\lambda}(y),\alpha_{\lambda}(z))$ then, $(A,[-,-],[-,-,-], )$ are multiplicative BiHom-Akivis algebras where the non-zero products are
$[e_1,e_2]=(\lambda+1)e_1,\ [e_2,e_1]=-\frac{1}{\lambda+1}e_1,\ [e_2,e_2]=
\frac{\lambda^2+2\lambda}{\lambda+1}e_1$ and $[e_1,e_2,e_2]=[e_2,e_2,e_2]=
\frac{1}{\lambda+1}e_1.$
\end{example}
\begin{defn}
 Let $( A , [-,-], [-,-,-], \alpha,\beta )$ and $(\tilde{A} , \{-,- \}, \{-,-, - \}, \tilde {\alpha}, \tilde{\beta})$ be biHom-Akivis algebras. A morphism $\phi :  A \rightarrow \tilde { A}$ of BiHom-Akivis algebras is a linear map of $\mathbb K$-modules $A$ and $\tilde { A}$ such that
 $f\circ\alpha=\tilde{\alpha}\circ f,   f\circ\beta=\tilde{\beta}\circ f$ 
 and
 \begin{eqnarray}
 &&  
 f\circ [-,-]=\{-,-\}\circ f^{\otimes 2} \mbox{, }
  f\circ [-,-,-]=\{-,-,-\}\circ f^{\otimes 3}\nonumber
 \end{eqnarray}
\end{defn}
For example, if take $( A , [-,-], [-,-,-], \alpha ,\beta)$ as a multiplicative BiHom-Akivis algebra, then the twisting self-maps $\alpha$ and $\beta$ are themselfs  endomorphisms of $(A , [-,-], [-,-,-])$.\\

The following result holds.
\begin{theorem}\label{thm1}
 Let $(A , [-,-], [-,-,-], \alpha,\beta )$ be a BiHom-Akivis algebra and $\varphi, \psi :  A \rightarrow  A$  self-morphisms of $( A , [-,-], [-,-,-], \alpha, \beta )$ such that $\varphi\psi=\psi\varphi$. Define on $A$ a bilinear operation $[-,-]_{\varphi,\psi}$ and a trilinear operation $[-,-,-]_{\varphi,\psi}$ by 

$[x,y]_{\varphi,\psi} := [\varphi(x),\psi (y)]$,

$[x,y,z]_{\varphi,\psi} := \varphi\psi^{2} ([x,y,z])$, for all $x,y,z \in A$.\\
 Then $A_{\varphi,\psi} := ( A, [-,-]_{\varphi,\psi} , [-,-,-]_{\varphi,\psi}, \varphi\alpha, \psi\beta )$ is a BiHom-Akivis algebra.

Moreover, if $( A , [-,-], [-,-,-], \alpha, \beta )$ is multiplicative, then $A_{\varphi,\psi}$ is also multiplicative. 
\end{theorem}
\begin{proof}
Clearly $[-,-]_{\varphi,\psi}$ (resp. $[-,-,-]_{\varphi,\psi}$) is a bilinear (resp. trilinear) map and the BiHom-skew-symmetry of $[-,-]$ in $( A , [-,-], [-,-,-], \alpha, \beta)$ implies the BiHom-skew-symmetry of 
$[-,-]_{\varphi,\psi}$ in $A_{\varphi,\psi}$.

Next, we have (by the Hom-Akivis identity (\ref{BiHAk})), 
\begin{eqnarray}
&& \circlearrowright_{x,y,z} [(\psi\beta)^2(x),[\psi\beta(y),\varphi\alpha(z)]_{\varphi,\psi} ]_{\varphi,\psi} = \circlearrowright_{x,y,z} [\varphi\psi^2\beta^2(x),\psi([\varphi\psi\beta(y),\psi\varphi\alpha(z)]) ] 
\nonumber\\
&&=\circlearrowright_{x,y,z} (\varphi\psi^2([\beta^2(x),[\beta(y),\alpha(z)]]))  = \circlearrowright_{x,y,z} (\varphi\psi^2([x,y,z]) - \varphi\psi^2([y,x,z]))\nonumber\\
&& = \circlearrowright_{x,y,z} ([x,y,z]_{\varphi,\psi} - [y,x,z]_{\varphi,\psi})\nonumber
\end{eqnarray}
The second assertion is proved as follows:
\begin{eqnarray}
&&\varphi\alpha([x,y]_{\varphi,\psi})=\varphi\alpha([\varphi(x),\psi(y)])
=[\varphi\alpha\varphi(x),\varphi\alpha\psi(y)]\nonumber\\
&&=[\varphi(\varphi\alpha(x)),\psi(\varphi\alpha(y))]=
[\varphi\alpha(x),\varphi\alpha(y)]_{\varphi,\psi},\nonumber\\
&&\psi\beta([x,y]_{\varphi,\psi})=\psi\beta([\varphi(x),\psi(y)])
=[\psi\beta\varphi(x),\psi\beta\psi(y)]\nonumber\\
&&=[\varphi(\psi\beta(x)),\psi(\psi\beta(y))]=
[\psi\beta(x),\psi\beta(y)]_{\varphi,\psi},\nonumber
\end{eqnarray}
and 
\begin{eqnarray}
&&\varphi\alpha([x,y,z]_{\varphi,\psi})=\varphi\alpha\varphi\psi^2([x,y,z])=\varphi\psi^2(
[\varphi\alpha(x),\varphi\alpha(y),\varphi\alpha(z)])=
[\varphi\alpha(x),\varphi\alpha(y),\varphi\alpha(z)]_{\varphi,\psi}\nonumber\\
&&\psi\beta([x,y,z]_{\varphi,\psi})=\psi\beta\varphi\psi^2([x,y,z])=\varphi\psi^2(
[\psi\beta(x),\psi\beta(y),\psi\beta(z)])=
[\psi\beta(x),\psi\beta(y),\psi\beta(z)]_{\varphi,\psi}\nonumber
\end{eqnarray}
This completes the proof.
\end{proof}
\begin{cor}
If $(A , [-,-], [-,-,-], \alpha, \beta)$ is a multiplicative BiHom-Akivis algebra, then so is $A_{\alpha^n,\beta^m}$ for all $n,m\in \mathbb{N}$.
\end{cor}
\begin{proof}
This follows from  Theorem \ref{thm1} if take  $\varphi = \alpha^n$ and $\psi=\beta^m$.
\end{proof}
\begin{cor}\label{corimp}
Let $( A , [-,-], [-,-,-])$ be an Akivis algebra and $\alpha, \beta$  endomorphisms of $(A , [-,-], [-,-,-])$. Define on $ A $ a bilinear operation $[-,-]_{\alpha,\beta}$ and a trilinear operation $[-,-,-]_{\alpha,\beta}$ by 
\par 
$[x,y]_{\alpha,\beta} := [\alpha (x), \beta (y)]$,
\par
$[x,y,z]_{\alpha,\beta} := \alpha\beta^2 ([x,y,z]))$, \\
for all $x,y,z \in A$ Then $A_{\alpha,\beta}=(A , [-,-]_{,\alpha,\beta}, [-,-,-]_{\alpha,\beta} , \alpha, \beta)$ is a multiplicative BiHom-Akivis algebra.
\par
Moreover, suppose that $ (B,\{-,-\},\{-,-,-\}$ is another Akivis algebra and that $\varphi, \psi$ are endomorphisms of $B$. If $f: A \rightarrow B$ is an Akivis algebra morphism satisfying $f\circ\alpha=\varphi\circ f$ and $f \circ \beta = \psi \circ f$, then $f: (A , [-,-]_{\alpha,\beta}, [-,-,-]_{\alpha,\beta} ,\alpha, \beta) \rightarrow (B , \{-,-\}_{\varphi,\psi}, \{-,-,-\}_{\varphi,\psi} , \varphi, \psi)$ is a morphism of multiplicative BiHom-Akivis algebras. \\
\end{cor}
\begin{proof}
The first of this theorem is a special case of Theorem\ref{thm1} above when $\alpha=\beta = id$. The second part is proved in the same way as in Theorem \ref{thm1}. For completeness, we repeat it as follows:
\begin{eqnarray}
&&f([x,y]_{\alpha,\beta})=f([\alpha(x),\beta(y)])=\{f\alpha(x),f\beta(y)\}=
\{\varphi f(x),\psi f(y)\}=\{f(x),f(y)\}_{\varphi,\psi}\nonumber\\
&& f([x,y,z]_{\alpha,\beta})= f\alpha\beta^2([x,y,z])=\varphi\psi^2(\{f(x),f(y),f(z)\})=\{f(x),f(y),f(z)\}_{\varphi,\psi}\nonumber
\end{eqnarray}
This completes the proof.
\end{proof}
\begin{example}
Consider the two-dimensional Akivis algebra $(A, [-,-],[-,-,-])$ with basis $(e_1, e_2 )$  given by
$$[e_1,e_2]=[e_1,e_2,e_2]=[e_2,e_2,e_2]=e_1$$
and all missing products are 0 (see Example 4.7 in \cite{issa1}).
For any $r,s\in \mathbb{R},$ the maps $\alpha_r$ and $\beta_s$ defined by 
$\alpha_r(e_1)=(r+1)e_1, \alpha_r(e_2)=re_1+e_2$ and 
$\beta_s(e_1)=(s+1)e_1, \beta_s(e_2)=se_1+e_2$ are commuting morphisms of $A.$ Note that $\alpha_r\neq\beta_s$ if and only if $r\neq s.$ Next, if we define 
the operations $[−, −]_{\alpha_r,\beta_s}$ and $[−, −, −]_{\alpha_r,\beta_s}$ with non-zero products by
$$[e_1,e_2]_{\alpha_r,\beta_s}=(r+1)e_1 \mbox{ and } [e_1,e_2,e_2]_{\alpha_r,\beta_s}=[e_2,e_2,e_2]_{\alpha_r,\beta_s}=(r+1)(s+1)^2e_1,$$
we get, by Corollary \ref{corimp}, that $A_{\alpha_r,\beta_s} =(A, [−, −]_{\alpha_r,\beta_s} , [−, −, −]_{\alpha_r,\beta_s} \alpha_r,\beta_s)$ are BiHom-Akivis algebras.
\end{example}
\section{BiHom-Malcev algebras from BiHom-Akivis algebras}
In this section we define BiHom-alternative and BiHom-flexible BiHom-Akivis 
algebras, and we give a characterization of BiHom-alternative algebras through
associate BiHom-Akivis algebras. The main result here is that the BiHom-Akivis algebra associated with a BiHom-alternative algebra has a BiHom-Malcev structure (this could be seen as another version of Theorem in \cite{chtioui1}).
\begin{defn}\label{alfl}
A BiHom-Akivis algebra ${\cal A} := (A, [-,-], [-,-,-], \alpha ,\beta)$ is said:
\par
(i) {\it BiHom-flexible}, if 
\begin{eqnarray}
[\alpha(x),\alpha(y),\alpha(z)]+[\alpha(z),\alpha(y),\alpha(x)] = 0
 \mbox{ for all $x,y \in  A$  } \label{BiHomfl}
\end{eqnarray}
(ii) {\it BiHom-alternative}, if 
\begin{eqnarray}
&& [\alpha^2(x),\alpha^2(y),\beta(z)]+[\alpha^2(y),\alpha^2(x),\beta(z)]=0
\mbox{ for all $x,y \in  A$  } \label{BiHomla}\\
&& [\alpha(x),\beta^2(y),\beta^2(z)]+[\alpha(x),\beta^2(z),\beta^2(y)]=0
\mbox{ for all $x,y \in  A$  }\label{BiHomra}
\end{eqnarray}
\end{defn}
\begin{remark}
\begin{enumerate}
\item The BiHom-flexible law (\ref{BiHomfl}) in ${\cal A}$ is equivalent to 
$ [\alpha(x),\alpha(y),\alpha(x)]=0$ for all $x,y\in A.$
\item The identities (\ref{BiHomla}) and  (\ref{BiHomra}) are respectively  called the left BiHom-alternativity  and the right alternativity. They are respectively equivalent to\\
 $ [\alpha^2(x),\alpha^2(x),\beta(y)]=0  \ and \  [\alpha(x),\beta^2(y),\beta^2(y)]=0 \ \ for \ all \ x, y \in {\cal A}$. 
\end{enumerate}
\end{remark}
The following result is an immediate consequence of Theorem \ref{thm0} and Definition \ref{alfl}.
\begin{proposition}\label{haf}
Let ${\cal A}=(A,\mu,\alpha,\beta)$ be a multiplicative regular BiHom-algebra and\\  ${\cal A_K}=(A, [-,-]=\mu-\mu\circ(\alpha^{-1}\beta\otimes\alpha\beta^{-1})\circ\tau, [-,-,-]=as_{\alpha,\beta}\circ(\alpha^{-1}\beta^2\otimes\beta\otimes\alpha), \alpha,\beta )$ its associate BiHom-Akivis algebra.
\begin{enumerate}
\item If $(A, \mu , \alpha,\beta )$ is BiHom-flexible, then ${\cal A_K}$ is BiHom-flexible.
\item If $(A, \mu , \alpha, \beta)$ is BiHom-alternative, then so is ${\cal A_K}$.
\end{enumerate}
\end{proposition}
We have the following characterization of BiHom-Lie algebras in terms of BiHom-Akivis algebras.
\begin{proposition}
Let ${\cal A} := (A, [-,-], [-,-,-], \alpha ,\beta)$ be a BiHom-flexible BiHom-Akivis algebra such that $\alpha$ is surjective. Then ${\cal A}_L=(A,[-,-],\alpha,\beta)$ is a BiHom-Lie algebra if, and only if $\circlearrowright_{x,y,z} [x,y,z] = 0$, for all $x,y,z \in  A.$
\end{proposition}
\begin{proof}
Pick $x, y, z$ in $A.$ Then there exists $a,b,c\in A$ such that 
$x=\alpha(a), y=\alpha(b), z=\alpha(c)$ since $\alpha$ is surjective. Therefore, the BiHom-Akivis identity (\ref{BiHAk}) and the BiHom-flexibility  (\ref{BiHomfl}) in ${\cal A}$ imply 
\begin{eqnarray}
&&\circlearrowright_{x,y,z}[\beta^2(x),[\beta(y),\alpha(z)]]=\circlearrowright_{x,y,z}[x,y,z]-\circlearrowright_{x,y,z}[y,x,z]
=\circlearrowright_{a,b,c}[\alpha(a),\alpha(b),\alpha(c)]\nonumber\\
&&-\circlearrowright_{a,b,c}[\alpha(b),\alpha(a),\alpha(c)]
=2\circlearrowright_{a,b,c}[\alpha(a),\alpha(b),\alpha(c)]=2\circlearrowright_{x,y,z} [x,y,z].\nonumber
\end{eqnarray}
Hence, $\circlearrowright_{x,y,z}[\beta^2(x),[\beta(y),\alpha(z)]]=0$ if and only if $\circlearrowright_{x,y,z}[x,y,z] = 0$ (recall that the ground field $\mathbb K$ is of characteristic 0).
\end{proof}
The following result is a slight generalization of Proposition 2.1 in \cite{chtioui1}, which in turn generalizes a similar well-known result in alternative rings. The reader can also see  Proposition 3.17 in \cite{YAU3} for the Hom-version of this result.
\begin{proposition}
Let ${\cal A}:= (A, [-,-], [-,-,-], \alpha, \beta )$ be a BiHom-alternative BiHom-Akivis algebra such that $\alpha$ and $\beta$ are surjective. Then 
\begin{eqnarray}
\circlearrowright_{x,y,z} [\beta^2(x),[\beta(y),\alpha(z)]]= 6 [x,y,z] \label{syl4}
\end{eqnarray}
for all $x,y,z \in A.$
\end{proposition}
\begin{proof}
Pick $x,y,z\in A.$ Then, there exists $a, b, c\in A$ such that 
$x=\alpha^2(a), y=\alpha^2(b), z=\alpha(c).$ Hence, the application to (\ref{BiHAk}) of the BiHom-alternativity in ${\cal A}$ gives :
\begin{eqnarray}
&&\circlearrowright_{x,y,z} [\beta^2(x),[\beta(y),\alpha(z)]]=\circlearrowright_{x,y,z} [x,y,z]-\sigma [y,x,z]=
\circlearrowright_{x,y,z} [x,y,z]-\circlearrowright_{x,y,z} [\alpha^2(b),\alpha^2(a),\beta(c)]\nonumber\\
&&\circlearrowright_{x,y,z} [x,y,z]-\circlearrowright_{x,y,z} [\alpha^2(a),\alpha^2(b),\beta(c)]=2\circlearrowright_{x,y,z} [x,y,z]
\end{eqnarray}
Next, again by  the BiHom-alternativity in ${\cal A}$ and surjectivity of $\alpha$ and $\beta$, we prove that $\sigma [x,y,z]=3[x,y,z].$ Therefore
\begin{eqnarray}
&&\circlearrowright_{x,y,z} [\beta^2(x),[\beta(y),\alpha(z)]]=6[x,y,z]\nonumber
\end{eqnarray}
\end{proof}
 First, let recall the following
\begin{defn}\cite{chtioui1} Let $(A,\mu,\alpha,\beta)$ be a regular BiHom-algebra.
Define the BiHom-Bruck-Kleinfeld function $f: A^{\otimes 4}\longrightarrow A$ as the multilinear map
\begin{eqnarray}
&&f(w, x, y, z)=as_{\alpha,\beta}(\beta^2(w)\alpha\beta(x), \alpha^2\beta(y),
 \alpha^3(z))-as_{\alpha,\beta}(\beta^2(x), \alpha\beta(y), \alpha^2(z))\alpha^3\beta(w)\nonumber\\
&&-\alpha^2\beta^2(x)as_{\alpha,\beta}(\alpha\beta(w), \alpha^2(y), \alpha^3\beta^{-1}(z)).
\end{eqnarray}
\end{defn}
The following result is very useful.
\begin{lemma}\cite{chtioui1} Let
$(A,\mu,\alpha,\beta)$ be a regular BiHom-alternative algebra.
Then the BiHom-Bruck-Kleinfeld function $f$ is alternating.
\end{lemma}
\begin{proposition}
 Let
$(A,\mu,\alpha,\beta)$ be a regular BiHom-alternative algebra. Then
\begin{eqnarray}
&& as_{\alpha,\beta}(\beta^3(x),\alpha\beta^2(y),\alpha\beta(x)\alpha^2(z))=as_{\alpha,\beta}(\alpha^{-1}\beta^3(x),\beta^2(y),\alpha\beta(z))\alpha^2\beta^2(x) \label{f1}\\
&& as_{\alpha,\beta}(\beta^3(x),\alpha\beta^2(y),\alpha\beta(z)\alpha^2(x))=\alpha\beta^3(x)as_{\alpha,\beta}(\beta^2(x),\alpha\beta(y),\alpha^2(z)) \label{f2}
\end{eqnarray}
for all $x,y,z\in A.$
\end{proposition}
\begin{proof}
For (\ref{f1}), we compute as follows
\begin{eqnarray}
&& as_{\alpha,\beta}(\beta^3(x),\alpha\beta^2(y),\alpha\beta(x)\alpha^2(z))
=as_{\alpha,\beta}(\beta^2(\beta(x)),\alpha\beta(\beta(y)),\alpha^2(\alpha^{-1}\beta(x)z))\nonumber\\
&&=as_{\alpha,\beta}(\beta^2(\alpha^{-1}\beta(x)z),\alpha\beta(\beta(x)),\alpha^2(\beta(y))) \ \ \mbox{(  by alternativity of $as_{\alpha,\beta}(\beta^2\otimes\alpha\beta\otimes\alpha^2)$ \ \  )}\nonumber\\
&&=as_{\alpha,\beta}(\beta^2(\alpha^{-1}\beta(x))\alpha\beta(\alpha^{-1}\beta(z)),\alpha^2\beta(\alpha^{-1}\beta(x)),\alpha^3(\alpha^{-1}\beta(y)))\nonumber\\
&&=f(\alpha^{-1}\beta(x),\alpha^{-1}\beta(z),\alpha^{-1}\beta(x),\alpha^{-1}\beta(y))\nonumber\\
&&+as_{\alpha,\beta}(\beta^2(\alpha^{-1}\beta(z)),\alpha\beta(\alpha^{-1}\beta(x)),\alpha^2(\alpha^{-1}\beta(y)))\alpha^3\beta(\alpha^{-1}\beta(x))\nonumber\\
&&+\alpha^2\beta^2(\alpha^{-1}\beta(z))as_{\alpha,\beta}(\alpha\beta(\alpha^{-1}\beta(x))
,\alpha^2(\alpha^{-1}\beta(x)), \alpha^3\beta^{-1}(\alpha^{-1}\beta(y)))\nonumber\\
&&=as_{\alpha,\beta}(\alpha^{-1}\beta^3(x),\beta^2(y),\alpha\beta(z))\alpha^2\beta^2(x)
\ \ \mbox{(  by alternativity of $f$ and $as_{\alpha,\beta}(\beta^2\otimes\alpha\beta\otimes\alpha^2).$ \ \  )}\nonumber
\end{eqnarray}
This finishes the proof of (\ref{f1}).

For  (\ref{f2}), we  compute as follows
\begin{eqnarray}
&& as_{\alpha,\beta}(\beta^3(x),\alpha\beta^2(y),\alpha\beta(z)\alpha^2(x))
=as_{\alpha,\beta}(\beta^2(\beta(x)),\alpha\beta(\beta(y)),\alpha^2(\alpha^{-1}\beta(z)x))\nonumber\\
&&=as_{\alpha,\beta}(\beta^2(\alpha^{-1}\beta(z)x),\alpha\beta(\beta(x)),\alpha^2(\beta(y))) \ \ \mbox{(  by alternativity of $as_{\alpha,\beta}(\beta^2\otimes\alpha\beta\otimes\alpha^2)$ \ \  )}\nonumber\\
&&=as_{\alpha,\beta}(\beta^2(\alpha^{-1}\beta(z))\alpha\beta(\alpha^{-1}\beta(x)),\alpha^2\beta(\alpha^{-1}\beta(x)),\alpha^3(\alpha^{-1}\beta(y)))\nonumber\\
&&=f(\alpha^{-1}\beta(z),\alpha^{-1}\beta(x),\alpha^{-1}\beta(x),\alpha^{-1}\beta(y))\nonumber\\
&&+as_{\alpha,\beta}(\beta^2(\alpha^{-1}\beta(x)),\alpha\beta(\alpha^{-1}\beta(x)),\alpha^2(\alpha^{-1}\beta(y)))\alpha^3\beta(\alpha^{-1}\beta(z))\nonumber\\
&&+\alpha^2\beta^2(\alpha^{-1}\beta(x))as_{\alpha,\beta}(\alpha\beta(\alpha^{-1}\beta(z))
,\alpha^2(\alpha^{-1}\beta(x)), \alpha^3\beta^{-1}(\alpha^{-1}\beta(y)))\nonumber\\
&&= \alpha\beta^3(x)as_{\alpha,\beta}(\beta^2(x),\alpha\beta(y),\alpha^2(z))
\ \ \mbox{(  by alternativity of $f$ and $as_{\alpha,\beta}(\beta^2\otimes\alpha\beta\otimes\alpha^2).$ \ \  )}\nonumber
\end{eqnarray}
This finishes the proof of (\ref{f2}).
\end{proof}
\begin{cor}
Let
$(A,\mu,\alpha,\beta)$ be a regular BiHom-alternative algebra. Then
\begin{eqnarray}
&&as_{\alpha,\beta}(\beta^3(x),\alpha\beta^2(y),[\alpha\beta(x),\alpha^2(z)])
=[as_{\alpha,\beta}(\alpha^{-1}\beta^3(x),\beta^2(y),\alpha\beta(z)),\alpha^2\beta^2(x)]\label{f3}
\end{eqnarray}
for all $x,y,z\in A$ where $[-,-]=\mu-\mu\circ(\alpha^{-1}\beta\otimes\alpha\beta^{-1})\circ\tau$
is the BiHom-commutator bracket
\end{cor}
\begin{proof}
Indeed, we have
\begin{eqnarray}
&&as_{\alpha,\beta}(\beta^3(x),\alpha\beta^2(y),[\alpha\beta(x),\alpha^2(z)])
\nonumber\\
&&=as_{\alpha,\beta}(\beta^3(x),\alpha\beta^2(y),\alpha\beta(x)\alpha^2(z))
-as_{\alpha,\beta}(\beta^3(x),\alpha\beta^2(y),\alpha\beta(z)\alpha^2(x))
\nonumber\\
&&=as_{\alpha,\beta}(\alpha^{-1}\beta^3(x),\beta^2(y),\alpha\beta(z))\alpha^2\beta^2(x)-
\alpha\beta^3(x)as_{\alpha,\beta}(\beta^2(x),\alpha\beta(y),\alpha^2(z))
\nonumber\\ && \mbox{( by (\ref{f1}) and (\ref{f2}) \ ) }\nonumber\\
&&=[as_{\alpha,\beta}(\alpha^{-1}\beta^3(x),\beta^2(y),\alpha\beta(z)),\alpha^2\beta^2(x)]\nonumber
\end{eqnarray}
as desired.
\end{proof}
We now come to the main result of this section, which is Theorem 2.2  in \cite{chtioui1} but from a point of view of BiHom-Akivis algebras.
\begin{theorem}
Let $(A, \mu , \alpha, \beta )$ be a regular BiHom-alternative BiHom-algebra and ${\cal A_K}=(A, [-,-]=\mu-\mu\circ(\alpha^{-1}\beta\otimes\alpha\beta^{-1})\circ\tau, [-,-,-]=as_{\alpha,\beta}\circ(\alpha^{-1}\beta^2\otimes\beta\otimes\alpha), \alpha,\beta )$ its associate BiHom-Akivis algebra. Then $(A, [-,-], \alpha,\beta)$ is a BiHom-Malcev algebra.
\end{theorem}
\begin{proof}
From Proposition \ref{haf} we get that ${\cal A_K}$ is Hom-alternative so that (\ref{syl4}) implies
\begin{eqnarray}
&&J_{\alpha,\beta}(\alpha\beta(x),\alpha\beta(y),[\beta(x),\alpha(y)])
=6[\alpha\beta(x),\alpha\beta(y),[\beta(x),\alpha(z)]]\nonumber\\
&&=6as_{\alpha,\beta}(\beta^3(x),\alpha\beta^2(y),[\alpha\beta(x),\alpha^2(z)])\nonumber\\
&&=[6as_{\alpha,\beta}(\alpha^{-1}\beta^3(x),\beta^2(y),\alpha\beta(z)),\alpha^2\beta^2(x)]
 \mbox{\ ( by (\ref{f3}) \ )}\nonumber\\
&&=[[\beta(x),\beta(y),\beta(z)],\alpha^2\beta^2(x)]=
[J_{\alpha,\beta}(\beta(x),\beta(y),\beta(z)),\alpha^2\beta^2(x)]\nonumber
\end{eqnarray}
and one recognizes the BiHom-Malcev identity (\ref{BiHomMalc}). Therefore, we get that $(A, [-,-], \alpha,\beta)$ is a BiHom-Malcev algebra.
\end{proof}
.
\end{document}